\documentclass[11pt, reqno]{amsart}

\usepackage{amssymb}
\usepackage{amsfonts}
\usepackage{amsmath}
\usepackage{bbm}
\usepackage[margin=1in]{geometry}
\usepackage{enumerate}
\usepackage{amsfonts}
\usepackage{amssymb}
\usepackage{graphicx}
\usepackage{mathrsfs}
\usepackage{hyperref}

\newtheorem{theorem}{Theorem}[section]
\theoremstyle{plain}

\newtheorem{claim}[theorem]{Claim}

\newtheorem{conjecture}[theorem]{Conjecture}

\newtheorem{lemma}[theorem]{Lemma}

\numberwithin{equation}{section}
\numberwithin{theorem}{section}
\numberwithin{case}{section}

\numberwithin{subcase}{case}
\renewcommand{\epsilon}{\varepsilon}
\numberwithin{equation}{section}

\def \e{\epsilon}

\def \T{\mathcal{T}}

\begin{document}
\title{Two-regular subgraphs of odd-uniform hypergraphs}
\author{Jie Han} 

\author{Jaehoon Kim}

\address{School of Mathematics, University of Birmingham, 
Edgbaston, Birmingham, B15 2TT, United Kingdom}
\email{J.Han@bham.ac.uk, j.kim.3@bham.ac.uk}

\thanks{The first author is supported by FAPESP (2013/03447-6, 2014/18641-5, 2015/07869-8). The second author is supported by the European Research Council under the European Union's Seventh Framework Programme (FP/2007--2013) / ERC Grant Agreements no. 306349 (J.~Kim).}
\maketitle

\begin{abstract}
Let $k\ge 3$ be an odd integer and let $n$ be a sufficiently large integer.
We prove that the maximum number of edges in an $n$-vertex $k$-uniform hypergraph containing no $2$-regular subgraphs is $\binom{n-1}{k-1} + \lfloor\frac{n-1}{k} \rfloor$, and the equality holds if and only if $H$ is a full $k$-star with center $v$ together with a maximal matching omitting $v$. This verifies a conjecture of Mubayi and Verstra\"{e}te.
\end{abstract}

\section{Introduction}\label{Introduction}

Tur\'an problems are central in extremal graph theory.
In general, Tur\'an-type problems question on the maximum number of edges of a (hyper)graph that does not contain certain subgraph(s). 
Their generalizations to hypergraphs appear to be extremally hard -- for example, despite many existing works, the Tur\'an density of tetrahedron (four triples on four vertices) is still unknown (see~\cite{Keevash11}).

Erd\H{o}s \cite{Erdos77} asked to determine the maximum size $f_k(n)$ of an $n$-vertex $k$-uniform hypergraph without any generalized 4-cycles, i.e., four distinct edges $A, B, C, D$ such that $A\cup B=C\cup D$ and $A\cap B=C\cap D=\emptyset$.
For $k=2$, this reduces to a well-known problem of studying the Tur\'an number for the 4-cycle.
It is known that $f_2(n)=(1+o(1))n^{3/2}$ \cite{Brown, ERS} and the exact value of $f_2(n)$ for infinitely many $n$ is obtained in \cite{Furedi83}.
For $k\ge 3$, F\"uredi \cite{Furedi} showed that $\binom{n-1}{k-1} + \lfloor\frac{n-1}{k} \rfloor\le f_k(n)\le \frac{7}{2}\binom{n}{k-1}$ and conjectured the following. \footnote{In fact, F\"uredi \cite{Furedi} found a slightly better lower bound for $k=3$, namely, $f_3(n)\ge \binom{n}{2}$ for $n\equiv 1$ or $5$ mod $20$.}
\begin{conjecture}\label{conj: furedi}
For $k\geq 4$ and $n\in \mathbb{N}$, $f_k(n) = \binom{n-1}{k-1} + \lfloor\frac{n-1}{k}$.
\end{conjecture}

The lower bound is achieved by a full $k$-star together with a maximal matching omitting its center.
Here a \emph{full $k$-star} is a $k$-uniform $n$-vertex hypergraph which consists of all $\binom{n-1}{k-1}$ sets of size $k$ containing a given vertex $v$, and the given vertex $v$ is called the {\em center} of the full $k$-star.
The most recent result on $f_k(n)$ is due to Pikhurko and Verstra\"{e}te~\cite{PV}, who showed that $f_k(n)\le \min\{1 + {2}/{\sqrt{k}},7/4\} \binom{n}{k-1}$, and $f_3(n)\le \frac{13}{9}\binom{n}{2}$. This improves a result by Mubayi and Verstra\"{e}te~\cite{MV2004}. In \cite{K}, the second author made a related conjecture about $k$-uniform hypergraphs containing no $r$ pairs of disjoint sets with the same union when $k$ is sufficiently bigger than $r$.

Since the generalized 4-cycles are 2-regular, i.e., each vertex has degree 2, one way to relax the original problem of Erd\H{o}s is to consider the maximum size of $n$-vertex (hyper)graphs without any 2-regular sub(hyper)graphs (or more generally, without any $r$-regular subgraphs).
In fact, the (relaxed) problem has its own interest even for graphs.
Although it is trivial for $r=2$, Pyber \cite{Pyber1985} proved that the largest number of edges in a graph with no $r$-regular subgraphs is $O(n\log{n})$ for any $r\geq 2$, and in \cite{PRS1995}, Pyber, R\"{o}dl and Szemer\'{e}di showed that there are graphs with no $r$-regular subgraphs having $\Omega(n\log\log{n})$ edges for any $r\geq 3$.

For non-uniform hypergraphs, it is easy to see that any hypergraph with no $r$-regular subgraphs has at most $2^{n-1}+r-1$ edges and Kostochka and the second author~\cite{KK} showed that if $n\geq \max\{425,r+1\}$ then any $n$-vertex hypergraph with no $r$-regular subgraphs having the maximum number of edges must contain a vertex of degree $2^{n-1}$.
For uniform hypergraphs, the problem becomes more interesting. 
One natural candidate for the extremal example of $k$-uniform hypergraphs with no $2$-regular subgraphs is the full $k$-star.
Indeed, Mubayi and  Verstra\"{e}te \cite{MV2009} proved the following.

\begin{theorem}\cite{MV2009}\label{thm:MV}
For every even integer $k\geq 4$, there exists $n_k$ such that the following holds for all $n\geq n_k$. If $H$ is an $n$-vertex $k$-uniform hypergraph with no 2-regular subgraphs, then
$|H|\leq \binom{n-1}{k-1}$.
Moreover, equality holds if and only if $H$ is a full $k$-star.
\end{theorem}

In \cite{K} the second author generalized the arguments in \cite{MV2009} and showed similar results for $k$-uniform hypergraphs with no $r$-regular subgraphs when $r\in \{3,4\}$.
Moreover, for odd $k$, Mubayi and Verstra\"{e}te \cite{MV2009} conjectured that $|H|\leq \binom{n-1}{k-1} + \left\lfloor \frac{n-1}{k} \right\rfloor$, and the only extremal graph is the full $k$-star plus a matching omitting its center.
In this paper, we prove this conjecture. 




\begin{theorem}\label{thm:main}
For every odd integer $k\geq 3$, there exists $n_k$ such that the following holds for all $n\geq n_k$. If $H$ is an $n$-vertex $k$-uniform hypergraph with no 2-regular subgraphs, then
$$|H|\leq \binom{n-1}{k-1} + \left\lfloor \frac{n-1}{k} \right\rfloor.$$
Moreover, equality holds if and only if $H$ is a full $k$-star with center $v$ together with a maximal matching omitting $v$.
\end{theorem}

Theorem~\ref{thm:MV} \cite{MV2009} is proved via the stability approach introduced by Erd\H{o}s and Simonovits~\cite{Simonovits}, which has been widely used in extremal set theory.
To prove Theorem~\ref{thm:main}, we also use the stability approach as well as some other ideas from \cite{MV2009}.
One advantage when $k$ is even is that there exist $2$-regular $k$-uniform hypergraphs on $3k/2$ vertices.
In contrast, for odd $k$, the smallest 2-regular $k$-uniform hypergraphs have order $2k$ and thus the analysis is more difficult (this is also the reason why more edges are allowed in the extremal graph for odd $k$, which makes the structure more complicated).
In our proof, we use some new tricks to overcome this difficulty.

\section{Preliminaries}

For a positive integer $N$ we write $[N]$ to denote the set $\{1,\dots,N\}$. 
We write $V(H)$ for the set of vertices, $E(H)$ for the set of edges in a hypergraph $H$. For a hypergraph $H$, we view $H$ as a collection of edges, thus sometimes $H$ refers to $E(H)$.
We say that $H$ is a {\em $k$-uniform hypergraph} or \emph{$k$-uniform family} if every edge of $H$ has size exactly $k$.
Moreover, we always say \emph{subgraph} instead of subhypergraph.
For a hypergraph $H$ and a set $S\subseteq V(H)$, 
$$N_{H}(S):=\{ e\setminus S : e\in E(H), S\subseteq e\}\enspace \text{ and }\enspace d_{H}(S)= |N_{H}(S)|.$$ We say a set $S$ is an {\em $s$-set} if $|S|=s$.
For a vertex $x\in V(H)$, we write $N_{H}(x):= N_{H}(\{x\})$ and $d_{H}(x):= d_{H}(\{x\})$.
We say $\{S,S'\}$ is an {\em equipartition} of a set $A$ if $|S|=|S'|$, $S\cap S'=\emptyset$ and $S\cup S'= A$.

 In order to prove Theorem~\ref{thm:main}, we use the following two theorems proved in \cite{MV2009}. These theorems give a rough structure of near-extremal hypergraphs.

\begin{theorem}\cite{MV2009} \label{thm:asymptotics}
For given $\epsilon>0$ and $k\in \mathbb{N}$, there exists $n_0 = n_0(k,\epsilon)$ such that the following holds for all $n\geq n_0$.
If $H$ is an $n$-vertex $k$-uniform hypergraph with no 2-regular subgraphs, then 
$$|H|\le (1+\epsilon)\binom{n-1}{k-1}.$$
\end{theorem}

\begin{theorem}\cite{MV2009} \label{thm:stability}
For given $\epsilon>0$ and $k\in \mathbb{N}$, there exists $n_1=n_1(k,\epsilon)$ such that the following holds for all $n\geq n_1$.
If $H$ is an $n$-vertex $k$-uniform hypergraph with no 2-regular subgraphs with $|H|\geq \binom{n-1}{k-1}$, then $H$ contains a vertex $v$ with $d_H(v)\geq (1-\epsilon)\binom{n-1}{k-1}$.
\end{theorem}

We use the following result of Frankl~\cite[Theorem 10.3]{Frankl87_survey}.
\begin{theorem}\label{ob: matching}
For integers $t\ge 1$ and $n\geq 2k$, if an $n$-vertex $k$-uniform hypergraph $H$ has more than $t{{n-1}\choose {k-1}}$ edges, then $H$ has a matching of size $t+1$.
\end{theorem}

We also use the following result of Balogh, Bohman and Mubayi \cite{BBM}.
If an intersecting $k$-uniform hypergraph is a subgraph of a full $k$-star, then it is called \emph{trivial}, otherwise \emph{non-trivial}.
Moreover, we say that a $k$-uniform hypergraph $H$ is covered by a set $X\subseteq \binom{V(H)}{2}$ of pairs of vertices of $H$ if for every hyperedge $e$ of $H$, there is a pair $\{x, y\} \in X$ such that $\{x,y\}\subseteq e$. 

\begin{lemma}\cite{BBM} \label{lem:BBM}
Let $H$ be a non-trivial intersecting $k$-uniform hypergraph.
Then $H$ can be covered by at most $k^2-k+1$ pairs of vertices.
\end{lemma}

\section{Proof of Theorem~\ref{thm:main}}
Let $k\ge 3$ be an odd integer. Let $\epsilon := \epsilon(k)>0$ be sufficiently small and let $n(k,\epsilon)$ be a sufficiently large integer. 
For $n\geq n(k,\epsilon)$, let $H$ be an $n$-vertex $k$-uniform hypergraph with no $2$-regular subgraphs. By removing edges if necessary, we may assume that
\begin{equation}\label{eq: |H| assumption}
|H|=\binom{n-1}{k-1} + \left\lfloor \frac{n-1}k \right\rfloor.
\end{equation}
 To prove Theorem~\ref{thm:main}, it is enough to show that $H$ contains a full $k$-star, because a full $k$-star with two additional intersecting edges always gives a $2$-regular subgraph. To derive a contradiction, we assume that $H$ does not contain any full $k$-star. Since $n$ is sufficiently large, Theorem~\ref{thm:stability} implies that there is a vertex $v\in V(H)$ such that
$d_{H}(v) \geq \binom{n-1}{k-1} - \epsilon^3 n^{k-1}$. 
Let $V':=V(H)\setminus \{v\}$, $H^*:=H[V']$ and $\tilde{H}:= \{e \setminus \{v\}: |e|=k, \, v\in e\notin H\}$.
Note that any $(k-1)$-set $A\subseteq V'$ with $A\notin \tilde{H}$ satisfies that $A\cup \{v\} \in H$.
Let
\begin{equation}\label{eq: x}
x:= |\tilde{H}|= \binom{n-1}{k-1} - d_H(v) \le \epsilon^3 n^{k-1}.
\end{equation}
Since $H$ does not contain a full $k$-star with center $v$, we have $x\geq 1$.
Then \eqref{eq: |H| assumption} and \eqref{eq: x} imply that 
\begin{equation}\label{eq: size of H*}
|H^*| = x + \left\lfloor \frac{n-1}k \right\rfloor . 
\end{equation}


Important idea for the proof is that a pair of two intersecting edges in $H^*$ ensures $\tilde{H}$ to contain more $(k-1)$-sets (Claim~\ref{cl: produce f}). 
Since $x>0$, there exists a pair of intersecting edges in $H^*$ and Claim~\ref{cl: produce f} implies that value of $x=|\tilde{H}|$ is larger. However, by \eqref{eq: size of H*}, larger value of $x$ guarantees more pairs of intersecting edges in $H^*$ which again implies the value of $x$ is larger. This circulation of logic gives a contradiction as $x$ cannot be too large by \eqref{eq: x}.

To turn this idea into a mathematical proof, we need to prove some technical claims.
Here, we give a brief outline of our proof.
We start with some simple but useful claims (Subsection 3.1), and in particular, we show that $x\ge n-2k+1$ (Claim~\ref{clm:n}).
Thus together with~\eqref{eq: x}, we may assume that there exists an integer $2\le \ell\le k-1$ such that 
\begin{equation*}
\epsilon^3 n^{\ell-1} \le x\le \epsilon^3 n^{\ell}.
\end{equation*}
We next find pairwise disjoint $(\ell-1)$-sets $S_1, S_2,\dots, S_{2k} \subseteq V'$ such that $d_{\tilde{H}}(S_i) \leq {\binom{k}{\ell-1} x}/{\binom{n-1}{\ell-1}}$ for $i\in [2k]$, which play an important role in the proof.
Let $\mathcal{T}:=\bigcup_{i=1}^{2k} N_{\tilde{H}}(S_i)$ be a collection of $(k-\ell)$-sets in $V'$.
Let $H_1:=\{e\in H^* : \exists T\in \T \text{ such that } T\subset e\}$ and let $H_0:=H^*\setminus H_1$. 
Our goal is to show that $|H_1|\le \e x$ (Subsection 3.2) and $|H_0| < (1-\e)x + \left\lfloor \frac{n-1}k \right\rfloor$ (Subsection 3.3), which together imply that $|H^*| = |H_1|+|H_0| < x+\left\lfloor \frac{n-1}k \right\rfloor$, contradicting~\eqref{eq: size of H*}.
In fact, the technical parts are Subsections 3.2 and 3.3, in which the essential argument is some clever double counting also used in~\cite{MV2009}.
However, as mentioned in Section 1, our case is more complicated than that in~\cite{MV2009}, so we have to proceed a more careful analysis (including introducing $\ell$ and $\mathcal{T}$) and use some new tricks (e.g. analyzing the intersecting property of certain family and utilizing Lemma~\ref{lem:BBM}).

\subsection{Preparation}
First we prove the following easy claim.

\begin{claim}\label{cl: produce f}
Assume we have $e_1,e_2\in H^*$, $A\subseteq V'$ and $\{ S,S' \}$ such that
\begin{itemize}
\item $A\cap (e_1\cup e_2) =\emptyset$,
\item $|A|=|e_1\cap e_2|-1$,
\item $\{S,S'\}$ is an equipartition of $e_1\triangle e_2$.
\end{itemize}
Then either $A\cup S \in \tilde{H}$ or $A\cup S' \in \tilde{H}$.
\end{claim}
\begin{proof}
If both $(k-1)$-sets $A\cup S$ and $A\cup S'$ are not in $\tilde{H}$, then 
$$e_1, e_2, A\cup S\cup \{v\} \text{ and } A\cup S'\cup \{v\}$$ form a $2$-regular subgraph of $H$, a contradiction. 
\end{proof}

Now we prove the following two claims regarding lower bounds on $x$.

\begin{claim}\label{cl: lower bound on H*}
Let $t \in [k-1]$. If $H^*$ contains two edges $e_1,e_2$ such that $|e_1\cap e_2|=t$, then $$x \ge \frac{1}{2}\binom{2k-2t}{k-t}\binom{n-2k+t-1}{t-1}.$$
\end{claim}
\begin{proof}
Suppose $e_1, e_2\in H^*$ such that $|e_1\cap e_2| = t$.
Consider a set $A \in \binom{V'\setminus(e_1\cup e_2)}{t-1}$ and an equipartition $\{S,S'\}$ of $e_1\triangle e_2$. For each $A$ and $\{S,S'\}$, Claim~\ref{cl: produce f} implies that $A\cup S \in \tilde{H}$ or $A\cup S' \in \tilde{H}$. Moreover, distinct choices of $(A,\{S,S'\})$ give us distinct $(k-1)$-sets in $\tilde{H}$.

Since there are $\binom{n-2k+t-1}{t-1}$ distinct choices of $A$ and  $\frac{1}{2}\binom{2k-2t}{k-t}$ distinct choices of $\{S,S'\}$, we have $x = |\tilde{H}| \geq \frac{1}{2}\binom{2k-2t}{k-t}\binom{n-2k +t -1}{t-1}$.
\end{proof}

\begin{claim}\label{clm:n}
The hypergraph $H^*$ contains two edges $e_1,e_2$ such that $|e_1\cap e_2| \geq 2$. Moreover, $x \ge n-2k+1$.
\end{claim}
\begin{proof}
Assume $H^*$ does not contain such two edges. Then for any $u\in V'$ and $S,S' \in N_{H^*}(u)$, we have $S\cap S'= \emptyset$. 
If there are two $(k-1)$-sets $S,S' \in N_{H^*}(u)$ such that $S,S'\notin \tilde{H}$, then
$$ S\cup \{u\}, S'\cup \{u\}, S\cup \{v\} \text{ and } S'\cup \{v\}$$
form a $2$-regular subgraph of $H$, a contradiction. 
Thus for any $u\in V'$, we have $|N_{H^*}(u)\cap \tilde{H}| \geq |N_{H^*}(u)|-1$. 
Moreover, by our assumption, we have $N_{H^*}(u)\cap N_{H^*}(u')=\emptyset$ for any distinct $u, u'\in V'$.
Thus
$$ x = |\tilde{H}| \geq \sum_{u \in V'} |N_{H^*}(u)\cap \tilde{H}|  \ge \sum_{u \in V'} (d_{H^*}(u)-1) = k |H^*| - (n-1) \stackrel{\eqref{eq: size of H*}}{\geq} k x - (k-1).$$
Since $k\geq 3$, we get $x\leq 1$.
However, the assumption that $x\geq 1$ and \eqref{eq: size of H*} imply that there are two edges $e_1,e_2 \in H^*$ with $|e_1\cap e_2|\geq 1$. So by Claim~\ref{cl: lower bound on H*}, we have $x \geq \frac{1}{2}\binom{2k-2}{k-1}\binom{n-2k}{0} \geq 3,$ a contradiction. Thus $H^*$ contains two edges $e_1,e_2$ with $|e_1\cap e_2|\geq 2.$ Hence Claim~\ref{cl: lower bound on H*} implies that $x\geq \frac{1}{2}\binom{2k-4}{k-2} \binom{n-2k +1}{1} \geq n-2k+1$.\end{proof}

By \eqref{eq: x} and Claim~\ref{clm:n}, there exists an integer $\ell$ such that 
\begin{equation}\label{eq: x1}
\epsilon^3 n^{\ell-1} \le x\le \epsilon^3 n^{\ell}
\end{equation}
and $2\leq \ell \leq k-1$. Throughout the rest of the paper, $\ell$ denotes such integer satisfying \eqref{eq: x1}.


The following claim finds $2k$ pairwise disjoint $(\ell-1)$-sets which have low degree in $\tilde{H}$.
\begin{claim}\label{clm:1}
There are pairwise disjoint $(\ell-1)$-sets $S_1, S_2,\dots, S_{2k} \subseteq V'$ such that $d_{\tilde{H}}(S_i) \leq {\binom{k}{\ell-1} x}/{\binom{n-1}{\ell-1}}$ for $i\in [2k]$.
\end{claim}
\begin{proof}
Let $F:= \{ S \in \binom{V'}{\ell-1} : d_{\tilde{H}}(S) \leq {\binom{k}{\ell-1} x}/{\binom{n-1}{\ell-1}} \}$ and $F':=\binom{V'}{\ell-1}\setminus F.$ 
So it suffices to find a matching of size $2k$ in $F$. Then 
$$ \binom{k-1}{\ell-1} x = \sum_{S \in \binom{V'}{\ell-1}} d_{\tilde{H}}(S) \geq 0\cdot |F| + \frac{\binom{k}{\ell-1} x}{\binom{n-1}{\ell-1}}|F'|.
$$
So we have
\[
|F'|\le \frac{\binom{k-1}{\ell-1}}{\binom{k}{\ell-1}} \binom{n-1}{\ell-1} = \frac{k-\ell+1}{k}\binom{n-1}{\ell-1}\le \frac{k-1}{k}\binom{n-1}{\ell-1},
\]
as $\ell\ge 2$.
Since $|F|+|F'|=\binom{n-1}{\ell-1}$, we have $|F| \ge \frac{1}{k}\binom{n-1}{\ell-1} > 2k\binom{|V'|-1}{\ell-2}$.
Then by Theorem~\ref{ob: matching},
$F$ contains a matching $\{S_1,\dots, S_{2k}\}$ of size $2k$ as desired.
\end{proof}

Let $S_1,\dots, S_{2k}$ be pairwise disjoint $(\ell-1)$-sets as in Claim~\ref{clm:1}.
Let $\mathcal{T}:=\bigcup_{i=1}^{2k} N_{\tilde{H}}(S_i)$ be a collection of $(k-\ell)$-sets in $V'$. So we have
\begin{equation}\label{eq: T}
|\T|\le \sum_{i=1}^{2k} | N_{\tilde{H}}(S_i)| \leq
\frac{2k\binom{k}{\ell-1} x}{\binom{n-1}{\ell-1}}  \le  2k^{k+1} n^{1-\ell} x.
\end{equation}
Note that for any $(k-\ell)$-set $T \notin \mathcal{T}$, we have $T\cup S_i\cup \{v\} \in H$ if $T\cap S_i=\emptyset$. 
Let $W=\bigcup_{T\in \T}T$, then
\begin{equation}\label{eq: size T}
|W| \leq (k-\ell)|\T| \stackrel{\eqref{eq: T}}{\leq} 2k^{k+2} n^{1-\ell} x \stackrel{\eqref{eq: x1}}{\leq} \epsilon^2 n.
\end{equation}
Let $H_1:=\{e\in H^* : \exists T\in \T \text{ such that } T\subset e\}$ and let $H_0:=H^*\setminus H_1$. 

We finish this subsection with an essential claim that bounds the degrees of vertex sets of size at most $\ell$ from above.

\begin{claim}\label{clm:Ffree}
Any $\ell$-set $L\subseteq V'$ satisfies that $|N_{H^*}(L)\setminus \T|\le 1$. Moreover, for any set $B\subseteq V'$ with $|B|=b\leq \ell$, it satisfies
$$d_{H_0}(B) \leq \binom{n-1-b}{\ell-b} \enspace \text{and} \enspace d_{H^*}(B)\le \binom{n-1-b}{\ell-b}(1+ |\T|).$$
\end{claim}
\begin{proof}
Suppose that there exists an $\ell$-set $L \subseteq V'$ such that $|N_{H^*}(L)\setminus \T|\geq 2$. Then there are two distinct $(k-\ell)$-sets $E_1,E_2 \in N_{H^*}(L)\setminus \T$. Since $2\leq \ell \leq k-1$, there exists $i\in [2k]$ such that $S_i$ is disjoint from $E_1\cup E_2 \cup L$. Also we choose a (possibly empty) set $A\subseteq V'\setminus W$ such that $|A|=|E_1\cap E_2|$ and $A\cap (E_1\cup E_2\cup L\cup S_i)=\emptyset$. This choice is possible since 
$$|V'\setminus (W\cup E_1\cup E_2\cup L\cup S_i)| \stackrel{\eqref{eq: size T}}{\geq} (n-1) - \epsilon^2 n - 2k + \ell -(\ell-1) \geq k.$$

We claim that both $ S_i\cup A \cup (E_1\setminus E_2)$ and $S_i\cup A\cup (E_2\setminus E_1)$ are not in $\tilde{H}$. 
Indeed, if $A=\emptyset$, then $E_j\setminus E_{3-j} = E_j$ for $j\in [2]$.
Since $E_j\notin \T=\bigcup_{i=1}^{2k} N_{\tilde{H}}(S_i)$, we obtain $ S_i\cup E_j= S_i \cup (E_j\setminus E_{3-j}) \notin \tilde{H}$ for $j\in [2]$. 
If $A\neq \emptyset$, then since $A\cap W=\emptyset$, we have for $j\in [2]$, $A\cup (E_j\setminus E_{3-j}) \not\subseteq W$.
So $A\cup (E_j\setminus E_{3-j}) \notin \T$ and thus $S_i\cup A\cup (E_j \setminus E_{3-j})\notin \tilde{H}$ for $j\in [2]$. Thus
$$ L\cup E_1, L\cup E_2, \{v\}\cup S_i\cup A \cup (E_1\setminus E_2) \text{ and }\{v\}\cup S_i\cup A\cup (E_2\setminus E_1)$$
form a $2$-regular subgraph of $H$, a contradiction. Thus the first part of the claim holds. 

For any $B\subseteq V'$ with $|B|=b\leq \ell$,
$$d_{H_0}(B) \leq \sum_{B\subseteq L, |L|=\ell} d_{H_0}(L) \le \sum_{B\subseteq L, |L|=\ell} |N_{H^*}(L)\setminus \T| \le \sum_{B\subseteq L, |L|=\ell} 1 = \binom{n-1-b}{\ell-b}.$$ 
Since $d_{H^*}(L) \leq |N_{H^*}(L)\setminus \T|+|\T|\le 1+|\T|$ for any $\ell$-set $L$, we also have
\[
d_{H^*}(B) \leq \sum_{B\subseteq L, |L|=\ell} d_{H^*}(L) \le \sum_{B\subseteq L, |L|=\ell} (1+|\T|) = \binom{n-1-b}{\ell-b}(1+ |\T|). \qedhere
\] 
\end{proof}

In the next two subsections, we show that $|H_1|\le \e x$ (Subsection 3.2) and $|H_0| < (1-\e)x + \left\lfloor \frac{n-1}k \right\rfloor$ (Subsection 3.3), which together imply that $|H^*| = |H_1|+|H_0| < x+\left\lfloor \frac{n-1}k \right\rfloor$.
This contradicts~\eqref{eq: size of H*} and thus completes the proof of Theorem~\ref{thm:main}.

\subsection{Size of $H_1$.}
In this subsection, we show that $|H_1|\leq \epsilon x$. 
We first consider the case $\ell \leq k-2$. 

\begin{claim}
If $\ell \leq k-2$, then $|H_1|\leq \epsilon x$.
\end{claim}
\begin{proof}

We first claim that we may assume that $|\T|>0$, $|H_1| \geq 3|\T|$ and $\ell \ge (k+1)/2$. 
Indeed, since $|\T|=0$ implies $|H_1| =0 \leq \epsilon x$, we may assume that $|\T|>0$.
If $|H_1|< 3|\T|$, then by \eqref{eq: T}, $|H_1| \leq 6k^{k+1} n^{1-\ell} x \leq \epsilon x$ because $\ell \geq 2$ and $n$ is sufficiently large. Thus we may assume that $|H_1|\geq 3|\T|$. 
Finally, since $|H_1| \geq 3|\T| >|\T|$, there is a $(k-\ell)$-set $T \in \T$ which is a subset of two distinct edges $e_1,e_2$ of $H_1$. Since $|e_1\cap e_2| \geq |T| \geq k-\ell$, Claim~\ref{cl: lower bound on H*} and \eqref{eq: x1} implies that 
$$ \frac{1}{2} \binom{2\ell}{\ell}\binom{n-k-\ell-1}{k-\ell-1} \leq x \stackrel{\eqref{eq: x1}}{\leq} \epsilon^3 n^{\ell}.$$
Since $n$ is sufficiently large and $\epsilon$ is small, this implies that $\ell > k-\ell-1$. Thus we have $\ell \ge (k+1)/2$ since $k$ is odd.

Let $p$ be the number of tuples $(T, \{e_1, e_2\}, f)$ with the following properties.
\begin{itemize}
\item[(P.1.1)] $T\in \T$, $\{e_1,e_2\} \in \binom{H_1}{2}$ and $f\in \tilde{H}$,
\item[(P.1.2)] $T\subseteq e_1\cap e_2$,
\item[(P.1.3)] $f\cap (e_1\cap e_2) =\emptyset$ and $\{ |f\cap e_1|, |f\cap e_2|\}=\{1, |e_2\setminus e_1|-1\}$.
\end{itemize}

First we find a lower bound on $p$. 
Fix a $(k-\ell)$-set $T$ in $\T$ and a pair $\{e_1,e_2\} \in P(T)$, where $P(T) := \{ \{e_1,e_2\} \in \binom{H_1}{2}: T\subseteq e_1\cap e_2\}$. Let $A$ be an arbitrary set of size $|e_1\cap e_2|-1$ in $V'\setminus (e_1\cup e_2)$ and let $\{S,S'\}$ be an equipartition of $e_1\triangle e_2$ such that $|S\cap e_1|=1 $.
Then Claim~\ref{cl: produce f} implies that one of $A\cup S$ and $A\cup S'$ belongs to $\tilde{H}$ and it satisfies (P.1.3). Note that distinct choices of $(A,\{S,S'\})$ give us distinct $(k-1)$-sets in $\tilde{H}$.

Note that $|e_1\cap e_2|\geq |T|=k-\ell$. Since there are at least $\binom{n-2k}{|e_1\cap e_2|-1}\ge \binom{n-2k}{k-\ell-1}$ distinct choices of $A$ and at least one choice of equipartition $\{S,S'\}$ with $|S\cap e_1|=1$, we obtain 
\begin{align*}
p&\ge \sum_{T\in \T} \sum_{\{e_1,e_2\} \in P(T)} \binom{n-2k}{k-\ell-1} =\binom{n-2k}{k-\ell-1}\sum_{T\in \T} \binom{d_{H_1}(T)}2 \nonumber \\
&\geq \binom{n-2k}{k-\ell-1} |\T| \binom{\frac1{|\T|}\sum_{T\in \T} d_{H_1}(T)}{2} \ge \binom{n-2k}{k-\ell-1} |\T| \binom{|H_1|/{|\T|}}{2}.
\end{align*}
Note that we get the penultimate inequality from the convexity of the real function $f(z)=\binom{z}{2}=z(z-1)/2$. 
Since $|H_1| \ge 3|\T|$, we have that $\binom{|H_1|/{|\T|}}{2} \ge |H_1|^2/(3{|\T|}^2)$ and thus
\begin{align}\label{eq: p1 lower bound}
 p\ge \frac{1}{3}\binom{n-2k}{k-\ell-1} \frac{|H_1|^2}{|\T|}.
 \end{align}

Now we find an upper bound of $p$. Clearly there are at most $x=|\tilde{H}|$ choices for the $(k-1)$-set $f$ and there are at most $|\T|$ choices for $T \in \T$. For given $f$, we choose two disjoint subsets $S_1,S_2\subseteq f$ with $|S_1|=1$. There are at most $(k-1)2^{k-2}$ ways to choose $S_1$ and $S_2$.

Assume that $f,T,S_1$ and $S_2$ are fixed, and we count the number of pairs of distinct edges $e_1,e_2 \in H_1$ such that $T\subseteq e_1\cap e_2$, $e_1\cap f=S_1, e_2\cap f=S_2$, and $|e_2\setminus e_1|-1=|e_2\cap f|=|S_2|$. 
We choose $e_1 \in H_1$ with $T\cup S_1\subseteq e_1$, and a set $B\subseteq e_1\setminus (T\cup S_1)$ with $|B|= k-|T|-|S_2|-1$. 
By Claim~\ref{clm:Ffree}, there are $d_{H_1}(T\cup S_1)\le \binom{n}{2\ell-k-1}(|\T|+1)$ ways to choose such an edge $e_1$ and there are at most $2^{k}$ ways to choose such a set $B$. 
Then we choose $e_2 \in H_1$ such that $T\cup B\cup S_2 \subseteq e_2$ and $e_1\cap e_2 = T\cup B$. 
There are at most $d_{H_1}(T\cup B\cup S_2)\le 1$ way to choose such a set $e_2$ by Claim~\ref{clm:Ffree}. 
Thus for fixed $f,T,S_1,S_2$, the number of choices of $e_1,e_2$ is at most $2^k \binom{n}{2\ell-k-1}(|\T|+1)$.
Thus we obtain
\begin{align}\label{eq: p1 upper bound}
 p &\leq \sum_{f\in \tilde{H}} \sum_{T\in \T}\sum_{S_1,S_2}2^k \binom{n}{2\ell-k-1}(|\T|+1) \nonumber \\
&\leq  x |\T| (k-1)2^{k-2} \cdot 2^{k}\binom{n}{2\ell-k-1}(|\T|+1) \leq   k 2^{2k} \binom{n}{2\ell-k-1} |\T|^2 x.
 \end{align}
 Note that the third sum is over $S_1,S_2$ satisfying $|S_1|=1, S_1\subseteq f, S_2\subseteq f\setminus S_1$.  From \eqref{eq: p1 lower bound} and \eqref{eq: p1 upper bound}, we get
\begin{eqnarray*}
|H_1|^2 \leq 3k2^{2k}   \binom{n}{2\ell-k-1} \binom{n-2k}{k-\ell-1}^{-1} |\T|^3 x\leq  k^{5k} n^{3\ell-2k}|\T|^3x \stackrel{\eqref{eq: T}}{\leq } k^{10k} n^{3-2k} x^4.
\end{eqnarray*}
Thus, we get
$$
|H_1|\leq k^{5k} n^{3/2-k} x^2 \stackrel{\eqref{eq: x1}}{\leq} k^{5k} \epsilon^3  n^{\ell+3/2-k} x \leq \epsilon x,
$$
because $\ell \leq k-2$.
\end{proof}

Now assume that $\ell=k-1$. 
In this case $\T$ is a collection of singletons, any vertex in $W$ belongs to $\T$ and $|W|=|\T|$.
We partition $H_1 = G_1 \cup G_2 \cup \cdots \cup G_k$, where $G_i=\{e\in H_1 \mid |e\cap W|=i\}$ for each $i\in [k]$.
Since $\ell=k-1$,  the fact that $\epsilon$ is small and \eqref{eq: T} imply
\begin{align}\label{eq: W size}
 3 |W|^{k-1} \leq 3(2k^{k+1} n^{2-k} x)^{k-1} \leq k^{2k^2} (x n^{1-k})^{k-2} x \stackrel{\eqref{eq: x1}}{\leq} k^{2k^2} (\epsilon^3 )^{k-2} x \leq \epsilon x/k.\end{align}
Now we show that $|G_i|\leq \epsilon x/k$ for all $i\in [k]$ which together imply that $|H_1|\le \e x$.

\begin{claim}\label{cl: k k-1}
$|G_k| \leq \epsilon x/k$ and $|G_{k-1}|\leq \epsilon x/k$.
\end{claim}
\begin{proof}
First, since $G_k$ does not contain any 2-regular subgraphs, by Theorem~\ref{thm:asymptotics}, there exists $n_0=n_0(k,2)$ such that if $|W|\ge n_0$ then $|G_k|\le 3\binom{|W|}{k-1}$.
If $|W|\le n_0$, then by \eqref{eq: x1} we have $|G_k|\le n_0^k < \e x/k$ since $n$ is large enough and $x \geq n-2k+1$ by Claim~\ref{clm:n}.
Otherwise $|W|>n_0$, then $|G_k|\le 3\binom{|W|}{k-1}\le \epsilon x/k$ by \eqref{eq: W size}.
Second, since $\ell=k-1$ we have
\[
|G_{k-1}| = \sum_{L \in  \binom{W}{k-1}} |N_{H^*}(L) \setminus \T|
\stackrel{{\rm Claim}~\ref{clm:Ffree}}{\leq } \sum_{L \in  \binom{W}{k-1}} 1 \leq |W|^{k-1}  \stackrel{\eqref{eq: W size}}{\leq } \epsilon x/k. \qedhere
\]
\end{proof}

Now we estimate $|G_i|$ for $1\le i\le k-2$.
\begin{claim}\label{clm:100}
$|G_i|\le \epsilon x/k$ for $i\in [k-2]$.
\end{claim}
\begin{proof}
Assume $|G_i|>\epsilon x/k$ for some $i\in[k-2]$, then \eqref{eq: W size} implies that $|G_i|\geq 3|W|^{k-1}$.
Let $p_i$ be the number of the tuples $(S, \{e_1, e_2\},f)$ with the following properties.
\begin{itemize}
\item[(P.2.1)] $e_1, e_2 \in G_i$ and $f\in \tilde{H}$,
\item[(P.2.2)] $S\in \binom{W}{i}$ and $S\subseteq e_1\cap e_2$,
\item[(P.2.3)] $f\cap (e_1\cap e_2) =\emptyset$, and $\{|f\cap e_1|,|f\cap e_2|\}=\{1, |e_1\setminus e_2|-1\}$.
\end{itemize}

Let $P_i(S):=\{ \{e_1,e_2\} \in \binom{G_i}{2} : S\subseteq e_1\cap e_2\}.$ By convexity, we have
\begin{align}\label{eq: convexity}
\sum_{S\in \binom{W}{i}} \sum_{ \{e_1, e_2\} \in P_i(S)} 1\ge \sum_{S\in \binom{W}{i}} \binom{d_{G_i}(S)}{2} \ge  \binom{|W|}{i} \binom{|G_i|/\binom{|W|}{i}}{2} \ge \frac{|G_i|^2}{3\binom{|W|}{i}},
\end{align}
where we used $\sum_{S\in \binom{W}{i}}d_{G_i}(S) = |G_i|$ and $|G_i| \geq 3|W|^{k-1} \geq 3\binom{W}{i}$.

Consider a set $S\subseteq W$ of size $i$, and a pair $\{e_1,e_2\} \in P_i(S)$. 
Let $A$ be an arbitrary set of size $|e_1\cap e_2|-1$ in $V'\setminus (e_1\cup e_2)$, and let $A_1,A_2$ be a partition of $e_1\triangle e_2$ such that $|A_1|=1$. 
The number of ways to choose $A$ is at least $\binom{n-2k}{|e_1\cap e_2|-1}\geq \binom{n-2k}{i-1}$ and the number of ways to choose $A_1,A_2$ is at least one. 
By Claim~\ref{cl: produce f}, at least one of $A\cup A_1 \in \tilde{H}$ and $A\cup A_2 \in \tilde{H}$ holds. 
Then either $(S,\{e_1,e_2\},A\cup A_1)$ or $(S,\{e_1,e_2\},A\cup A_2)$ satisfies (P.2.1)--(P.2.3). 
Since distinct choices of $(S,\{e_1,e_2\}, A, \{A_1,A_2\})$ give us distinct tuples, we have
\begin{align}\label{eq: p2 lower bound}
p_i\geq \sum_{S\in \binom{W}{i}} \sum_{\{e_1, e_2\}\in P_i(S)} \sum_{A} \sum_{A_1,A_2} 1 \geq \sum_{S\in \binom{W}{i}} \sum_{\{e_1, e_2\}\in P_i(S)} \binom{n-2k}{i-1} 
\stackrel{\eqref{eq: convexity}}{\geq}  \binom{n-2k}{i-1} \frac{|G_i|^2}{3\binom{|W|}{i}}.
\end{align}

Now we find an upper bound of $p_i$. Clearly there are at most $x=|\tilde{H}|$ choices of $f \in \tilde{H}$ and at most $\binom{|W|}{i}$ choices of $S\in \binom{W}{i}$ with $S\cap f =\emptyset$. For given $f$ and $S$, we choose two disjoint sets $A_1,A_2\subseteq f$ with $|A_1|=1$. There are at most $(k-1)2^{k-2}$ ways to choose such $A_1$ and $A_2$.

Assume $f,S,A_1,A_2$ are given, and we count the number of pairs $\{e_1,e_2\} \in P_i(S)$ such that $e_1\cap f=A_1$, $e_2\cap f=A_2$ and $|e_2\setminus e_1|-1=|e_2\cap f|=|A_2|$.  
We choose $e_1 \in G_i$ such that $S\cup A_1\subseteq e_1$ and $e_1 \setminus (S\cup A_1) \subseteq V'\setminus W$, and the number of ways to choose such $e_1$ is at most
\begin{eqnarray*}
\sum_{S\cup A_1\subseteq L\in \binom{V'}{k-1}} |N_{H^*}(L)\setminus W| \stackrel{\textrm{Claim}~\ref{clm:Ffree}}{\leq} \sum_{S\cup A_1\subseteq L\in \binom{V'}{k-1}}1 = \binom{n-i-2}{k-i-2}.
\end{eqnarray*}
We also choose a set $B\subseteq e_1\setminus (S\cup A_1)$ with $|B|= k-|S|-|A_2|-1$. There are at most $2^{k}$ ways to choose such a set $B$.
Then we choose $e_2 \in G_i$ such that $S\cup B\cup A_2 \subseteq e_2$, $e_1\cap e_2 = S\cup B$ and $e_2\setminus (S\cup B\cup A_2) \subseteq V'\setminus W$, and the number of ways to choose such $e_2$ is at most $|N_{H^*}(S\cup B\cup A_2)\setminus W|\le 1$ by Claim~\ref{clm:Ffree}.
Overall, for fixed $f,S,A_1,A_2$, the number of choices of $e_1,e_2$ is at most $2^k \binom{n-i-2}{k-i-2}$.
Thus we obtain
\begin{eqnarray}\label{eq: p2 upper bound}
 p_i &\leq& \sum_{f\in \tilde{H}} \sum_{S \in \binom{W}{i}}\sum_{A_1,A_2}2^k \binom{n}{k-i-2}\nonumber \\
&\leq & x \binom{|W|}{i} (k-1)2^{k-2} \cdot 2^{k}\binom{n}{k-i-2} \leq   k 2^{2k} x \binom{|W|}{i}\binom{n}{k-i-2} .
 \end{eqnarray}
 Note that the third sum is over $A_1,A_2$ satisfying $|A_1|=1, A_1\subseteq f, A_2\subseteq f\setminus A_1$. 
  
From \eqref{eq: p2 lower bound} and \eqref{eq: p2 upper bound} and the fact that $|W|=|\T|$, we obtain
\begin{eqnarray*}
|G_i|^2 &\leq& 3 k 2^{2k} \binom{|W|}{i}^{2}  \binom{n}{k-i-2} \binom{n-2k}{i-1}^{-1} x \\
&\stackrel{\eqref{eq: T}}{\leq}& k^{3k} (2 k^{k+1} n^{2-k} x)^{2i} n^{k-2i-1} x \\
&\leq & k^{10 k^2} n^{-(2i-1)(k-1)} x^{2i+1}\\
&\stackrel{\eqref{eq: x1}}{\leq} & k^{10 k^2} n^{-(2i-1)(k-1)} (\epsilon^3 n^{k-1})^{2i-1} x^2 \leq  k^{10k^2}  \epsilon^3 x^2 < \epsilon^2 x^2/k^2.
\end{eqnarray*}
This contradicts that $|G_i|> \epsilon x/k$. 
Thus the claim holds.
\end{proof}


\subsection{Size of $H_0$.}
At last we show that $|H_0| < (1-\e)x + \left\lfloor \frac{n-1}k \right\rfloor$.
Assume to the contrary, that 
\begin{align}\label{eq: H0 size}
 |H_0| \geq (1-\epsilon)x + \left\lfloor \frac{n-1}k\right\rfloor.
\end{align}
For any $u\in V'$, let $F_u:= N_{H_0}(u) \setminus \tilde{H}$.
We first observe that $F_u$ is an intersecting family.

\begin{claim}\label{cl: intersecting}
For any $u\in V'$, $F_u$ forms an intersecting family.
\end{claim}

\begin{proof}
If not, then there are two disjoint $(k-1)$-sets $A,A' \in F_u= N_{H_0}(u)\setminus \tilde{H}$. Since $A,A'\notin \tilde{H}$, 
$$ A\cup \{u\}, A'\cup \{u\}, A\cup \{v\} \text{ and } A'\cup \{v\}$$
form a $2$-regular subgraph of $H$, a contradiction. 
\end{proof}

If $f\in \tilde{H}$ belongs to $N_{H_0}(u)\cap N_{H_0}(u')$ for distinct $u, u' \in V'$, then fix any $L\in \binom{f}{\ell}$, we have $d_{H_0}(L)\geq 2$, contradicting Claim~\ref{clm:Ffree}. Thus $f\in \tilde{H}$ belongs to $N_{H_0}(u)$ for at most one $u\in V'$, which implies that
\begin{equation}\label{eq:H0}
x= |\tilde{H}| \ge \sum_{u\in V'} |N_{H_0}(u)\cap \tilde{H}| =  \sum_{u\in V'} (d_{H_0}(u)-|F_u|).
\end{equation}

Note that this implies that $\ell \geq 3$ and $k\geq 5$. 
In fact, if $\ell=2$, then Claim~\ref{clm:Ffree} implies that $d_{H_0}(\{u,u'\})\leq 1$ for any two distinct vertices $u,u'\in V'$, i.e., any two edges in $H_0$ share at most one vertex. Thus $N_{H_0}(u)$ forms a matching. By Claim~\ref{cl: intersecting}, $|F_u|\le 1$ as $F_u$ is an intersecting subfamily of a matching.
By \eqref{eq:H0},
$$x \geq \sum_{u\in V'} (d_{H_0}(u)-1) = k|H_0| - (n-1) \stackrel{\eqref{eq: H0 size}}{\geq} k(1-\epsilon)x - (k-1).$$
Thus $x \leq 1$, contradicting \eqref{eq: x1}. Thus $\ell \geq 3$. Since $3\leq \ell \leq k-1$ by \eqref{eq: x1} and $k$ is odd, we have $k\geq 5$.


Let 
$$X:=\{u\in V': F_u \text{ is a trivial intersecting family} \}$$ and for $u\in X$, let $p(u)$ be a vertex in $V'$ such that every $(k-1)$-set in $F_u$ contains $p(u)$.
We claim that
\begin{equation}\label{eq:sumX}
\sum_{u \in X} |F_u| \geq (1-\epsilon)(k-1)|H_0|.
\end{equation}
We first show that for $u\notin X$, $|F_u| \leq k^2 n^{\ell-3}$.
Indeed, since $u\notin X$, $F_u$ is a non-trivial intersecting $(k-1)$-uniform family. By Lemma~\ref{lem:BBM}, there are pairs of vertices $w_1w'_1,\dots, w_{t}w'_{t}$ with $t\leq (k-1)^2-(k-1)+1 \leq k^2$ which together cover all $(k-1)$-sets in $F_u$. 
Since $\ell\geq 3$, Claim~\ref{clm:Ffree} implies
\[
|F_u|\leq \sum_{i=1}^{t} d_{H_0}( \{u,w_i,w'_i\} ) \leq k^2 \binom{n-4}{\ell-3} \leq k^2 n^{\ell-3}.
\]
Then note that $\sum_{u\in V'} d_{H_0}(u) = k|H_0|$. From \eqref{eq:H0}, we get 
\begin{eqnarray*}
x&\geq& \sum_{u\in V'} (d_{H_0}(u) - |F_u|) \geq \sum_{u\in V'} d_{H_0}(u) -  \sum_{u\in X} |F_u| -\sum_{u\in V'\setminus X} |F_u|\\ &\geq& k|H_0| -  \sum_{u\in X} |F_u| - \sum_{u\in V'\setminus X} k^2 n^{\ell-3}
\geq  k|H_0| - \sum_{u\in X} |F_u| - k^2 n^{\ell-2}.
\end{eqnarray*}
Since $n$ is sufficiently large, \eqref{eq: x1} implies that $k^2 n^{\ell-2} \leq \epsilon^4 n^{\ell-1} \leq \epsilon x$.
Thus we get
\begin{align*}
\sum_{u\in X} |F_u| \geq k|H_0| -x - \epsilon x \stackrel{\eqref{eq: H0 size}}{\geq} k|H_0| -  \frac{(1+\epsilon)|H_0|}{1-\epsilon}  \geq (1-\epsilon)(k-1)|H_0|
\end{align*}
as $k\ge 5$.
So~\eqref{eq:sumX} is proved.

%
%

For $t\in [k-1]$, let $q_t$ be the number of the tuples $(u, \{e_1, e_2\}, f)$ with the following properties.
\begin{enumerate}[($i$)]
\item[(Q1)$_t$] $u\in X$, $e_i\setminus \{u\} \in F_u$ for $i\in [2]$ and $|e_1\cap e_2|=t$,
\item[(Q2)] $f \in \tilde{H}$, $f\cap (e_1\cap e_2)=\emptyset$,
\item[(Q3)] $\{|f\cap e_1|, |f\cap e_2|\} = \{ 1, |e_1\setminus e_2|-1\}$.
\end{enumerate}

For $u\in X$ and $t\in [k-1]$, we let 
$$F^t_u:=\{\{e_1,e_2\}: e_i\setminus\{u\} \in F_u \text{ for }i\in [2], |e_1\cap e_2|=t\},$$
and
$$ P^t:= \{(u,\{e_1,e_2\}): u\in X, \{e_1,e_2\} \in F^t_u\}.$$
Note that $F^1_u=\emptyset$ for any $u\in X$ since $F_u$ is an intersecting family. Since $u\in X$, we have $\{u, p(u)\} \subseteq e_1\cap e_2$ for $(u,\{e_1,e_2\}) \in P^t$.
By convexity of function $f(z)=\binom{z}{2} = \frac{z(z-1)}{2}$, we have
\begin{eqnarray}\label{eq:sum}
\sum_{t=2}^{k-1} |P^t| &=& \sum_{u\in X}\binom{|F_u|}{2} \ge |X| \binom{\frac{1}{|X|}\sum_{u\in X}|F_u|}{2} \nonumber \\
&\stackrel{\eqref{eq:sumX}}{\ge}& \frac{(1-\epsilon)^2(k-1)^2|H_0|^2}{2n} - \frac{1}{2}(1-\epsilon)(k-1)|H_0| \ge  \frac{(1-2\epsilon)(k-1)^2|H_0|^2}{2n}.
\end{eqnarray}
Here, we get the last inequality since we have $\e^2 |H_0| \geq \e^2 x \geq \epsilon^5 n^2 > 2n$ from \eqref{eq: x1}, the fact that $\ell \ge 3$ and $n$ is large.

Now we find a lower bound of $q_t$. Note that $q_1=0$ since $F_u$ is an intersecting family for any $u\in X$.
For $2\le t\le k-1$, first fix a vertex $u\in X$ and let $\{e_1,e_2\} \in F^t_u$.
We choose a set $A \subseteq V'\setminus (e_1\cup e_2)$ of size $t-1$.
We also choose an equipartition $\{S,S'\}$ of $e_1\triangle e_2$ such that $|S\cap e_1|=1$. The number of choices of such $\{S,S'\}$ is $(k-t)^2$.
Then Claim~\ref{cl: produce f} implies that either $A\cup S$ or $A\cup S'$ belongs to $\tilde{H}$ and it satisfies (Q3) as it plays the role of $f$. Note that for distinct choices of $(A,\{S,S'\})$, we get distinct $(k-1)$-sets $f$ in $\tilde{H}$.

So for $2\leq t\leq k-1$,
\begin{eqnarray*}
q_t &\ge& \sum_{u\in X} \sum_{\{e_1,e_2\}\in F^t_u} (k-t)^2\binom{n-2k}{t-1} = (k-t)^2 \binom{n-2k}{t-1} |P^t|.
\end{eqnarray*}
Since $k\geq 5$ and $n$ is large, $(k-t)^2 \binom{n-2k}{t-1}\geq n(n-2k)$ for $t\geq 3$. Thus we obtain
\begin{align}\label{eq: qt lower bound}
q_2 \geq (k-2)^2 (n-2k) |P^2| \enspace \text{ and }\enspace
q_t \geq n(n-2k) |P^t| \enspace \text{for} \enspace 3\le t\le k-1.
\end{align}

Next we find an upper bound of $q_t$.
Clearly there are at most $x=|\tilde{H}|$ choices for $f\in \tilde{H}$. 
We choose two disjoint sets $A_1, A_2 \subseteq f$ such that $|A_1|=1, |A_2|=k-t-1$. The number of such choices is at most $(k-1)\binom{k-2}{k-t-1}$. Now we choose $e_1$ in $H_0$ containing $A_1$. The number of choices for $e_1$ is at most $|H_0|$. Once $e_1$ is chosen, we choose $u\in X\cap (e_1 \setminus A_1)$ such that $\{u,p(u)\}\subseteq e_1$. There are at most $k-1$ such choices for $u$. 
Now we choose a $(t-2)$-subset $B\subseteq e_1\setminus (A_1\cup \{u,p(u)\})$, and there are $\binom{k-3}{t-2}$ ways to choose such $B$.
For given $A_2,B, u$, we choose $e_2$ such that $A_2\cup B\cup \{u,p(u)\} \subseteq e_2$.
Since $|A_2\cup B\cup \{u,p(u)\}|=k-1$, it contains a subset $L$ of size  $\ell$, thus Claim~\ref{clm:Ffree} implies that $d_{H_0}(A_2\cup B\cup \{u,p(u)\}) \leq d_{H_0}(L) \leq 1$. Thus the number of choices of $e_2$ is at most $1$.
Thus we get
\begin{align}\label{eq: qt upper bound}
q_t \leq x (k-1)\binom{k-2}{k-t-1}|H_0| (k-1) \binom{k-3}{t-2}.
\end{align}
Thus we obtain
\begin{eqnarray*}
(k-2)^2 n(n-2k) \sum_{t=2}^{k-1} |P^t| &\stackrel{\eqref{eq: qt lower bound}}{\leq}& n q_2 + (k-2)^2 \sum_{t=3}^{k-1} q_t \\
&\stackrel{\eqref{eq: qt upper bound}}{\leq}&  (k-1)^2(k-2) xn|H_0| + \sum_{t=3}^{k-1} (k-1)^4 \binom{k-2}{k-t-1}\binom{k-3}{t-2} x|H_0| \\
&\leq & (1+\epsilon)  (k-1)^2(k-2) xn|H_0|.
\end{eqnarray*}
Note that we get the last inequality since $n$ is sufficiently lage. By \eqref{eq:sum}, we get
$$ \frac{1}{2}(1-2\epsilon)(k-2)^2(k-1)^2|H_0|^2(n-2k) \leq (k-2)^2 n(n-2k) \sum_{t=2}^{k-1} |P^t| \leq (1+\epsilon)(k-1)^2(k-2)x n|H_0|.$$
Since $|H_0|>0$, dividing both sides by $(k-2)(k-1)^2 |H_0|/2$ gives
$$(1-2\epsilon)(k-2)|H_0|(n-2k) \leq 2(1+\epsilon) xn.$$
Since we have $(1-\epsilon)x \leq |H_0|$ from \eqref{eq: H0 size},
$$ (1-2\epsilon)(k-2)(1-\epsilon) x(n-2k) \leq 2(1+\epsilon)x n.$$
Since $x\geq 1$, we get $ (1-3\epsilon)(k-2)(n-2k)\leq 2(1+\epsilon)n$,
which is a contradiction since $k\geq 5$, $\epsilon$ is small and $n$ is large enough. 
So~\eqref{eq: H0 size} does not hold and we are done.

\section{Concluding Remarks}

In our proof of Theorem~\ref{thm:main}, except the use of Theorems~\ref{thm:asymptotics} and~\ref{thm:stability}, we only use the assumption that $H$ does not contain any $2$-regular subgraphs on $2k$ vertices. This motivates the following conjecture.

\begin{conjecture}\label{conj:1}
For every integer $k\geq 3$, there exists $n_k$ such that the following holds for all $n\geq n_k$. If $H$ is an $n$-vertex $k$-uniform hypergraph with no 2-regular subgraphs on $2k$ vertices, then
$$|H|\leq \binom{n-1}{k-1} + \left\lfloor \frac{n-1}{k} \right\rfloor.$$
Moreover, equality holds if and only if $H$ is a full $k$-star with center $v$ together with a maximal matching omitting $v$.
\end{conjecture}

For $k\ge 4$, Conjecture~\ref{conj: furedi} implies Conjecture~\ref{conj:1}.
Note that Conjecture~\ref{conj:1} stands betweenConjecture~\ref{conj: furedi} and the result on forbidding $2$-regular subgraphs.
In some sense it is more close to Conjecture~\ref{conj: furedi} -- because only finitely many (independent of $n$) configurations are forbidden (in contrast, by forbidding all $2$-regular subgraphs, the number of instances forbidden is related to $n$).
By our proof, to show Conjecture~\ref{conj:1} for odd integers $k$, it suffices to prove an asymptotical result and a stability result.

In this paper we focused on forbidding $2$-regular subgraphs.
It is natural to consider hypergraphs without $r$-regular subgraphs for $r\ge 3$ (see Question~6.9 in \cite{K}). We remark that Construction~6.8 in \cite{K} gives a lower bound on the maximum number of edges in such a hypergraph.

\end{document}